\documentclass[a4paper,12pt]{amsart} 
\title{Logarithmic moduli of roots of line bundles on curves}

\usepackage{amsmath, amssymb, mathrsfs, amsthm, shorttoc, geometry, stmaryrd, comment}
\usepackage{mathtools} 
\usepackage{hyperref}
\usepackage{xcolor}
\hypersetup{
colorlinks,
linkcolor={black},
citecolor={blue!50!black},
urlcolor={blue!80!black}
}

\usepackage[all]{xy}

\usepackage{tabularx}
\usepackage{longtable}
\numberwithin{equation}{subsection}

\usepackage{enumitem}

\usepackage{cleveref}
\crefformat{equation}{(#2#1#3)}

\AtBeginDocument{\renewcommand{\ref}[1]{\cref{#1}}}

\usepackage{pgf,tikz}
\usetikzlibrary{matrix, calc, arrows}

\usepackage{tikz-cd} 

\usepackage{xparse}  
\usetikzlibrary{decorations.pathmorphing}

\makeatletter
\newcommand{\math@param}[3]{%
  \fontdimen#3
  \ifx#1\displaystyle\textfont#2
  \else\ifx#1\textstyle\textfont#2
  \else\ifx#1\scriptstyle\scriptfont#2
  \else\scriptscriptfont#2 \fi\fi\fi
}

\newdimen\@my@yshift
\NewDocumentCommand{\myvec}{m}{\mathord{\mathpalette\@myvec{#1}}}

\newcommand*{\@myvec}[2]{%
  \begin{tikzpicture}[baseline=(n.base)]
  \node (n) [inner sep=0]
    {\global\@my@yshift=-0.5\math@param{#1}{2}{5}%
     $\m@th #1#2$};
    \draw[decorate, decoration={snake, amplitude=0.3pt, segment length=2pt}]
      ([yshift=\@my@yshift]n.south west) -- ([yshift=\@my@yshift]n.south east);
  \end{tikzpicture}%
}
\makeatother

\makeatletter
\newcommand*{\doublerightarrow}[2]{\mathrel{
  \settowidth{\@tempdima}{$\scriptstyle#1$}
  \settowidth{\@tempdimb}{$\scriptstyle#2$}
  \ifdim\@tempdimb>\@tempdima \@tempdima=\@tempdimb\fi
  \mathop{\vcenter{
    \offinterlineskip\ialign{\hbox to\dimexpr\@tempdima+1em{##}\cr
    \rightarrowfill\cr\noalign{\kern.5ex}
    \rightarrowfill\cr}}}\limits^{\!#1}_{\!#2}}}
\newcommand*{\triplerightarrow}[1]{\mathrel{
  \settowidth{\@tempdima}{$\scriptstyle#1$}
  \mathop{\vcenter{
    \offinterlineskip\ialign{\hbox to\dimexpr\@tempdima+1em{##}\cr
    \rightarrowfill\cr\noalign{\kern.5ex}
    \rightarrowfill\cr\noalign{\kern.5ex}
    \rightarrowfill\cr}}}\limits^{\!#1}}}
\makeatother

\newcommand{\on}[1]{\operatorname{#1}}
\newcommand{\bb}[1]{{\mathbb{#1}}}

\newcommand{\ca}[1]{{\mathcal{#1}}}
\newcommand{\bd}[1]{{\mathbf{#1}}}

\newcommand{\ul}[1]{{\underline{#1}}}

\def\sheafhom{\mathcal{H}om}

\newcommand{\sub}{\subseteq}

\theoremstyle{definition}
\newtheorem{definition}{Definition}[section]

\theoremstyle{plain}

\newtheorem{proposition}[definition]{Proposition}
\newtheorem{lemma}[definition]{Lemma}
\newtheorem{theorem}[definition]{Theorem}
\newtheorem{corollary}[definition]{Corollary}

\theoremstyle{remark}

\newtheorem{remark}[definition]{Remark}

\newtheorem{example}[definition]{Example}

\usepackage{letltxmacro}
\LetLtxMacro{\phiorig}{\phi}
\renewcommand{\phi}{\varphi}

\newcommand{\Gcomment}[1]{{\color{purple}G: #1}}

\author{David Holmes and Giulio Orecchia}
\date{\today}

\newcounter{nootje}
\newcommand{\beq}{\begin{equation}}
\newcommand{\eeq}{\end{equation}}
\newcommand{\beqs}{\begin{equation*}}
\newcommand{\eeqs}{\end{equation*}}
\renewcommand{\k}{k}
\tikzset{
  symbol/.style={
    draw=none,
    every to/.append style={
      edge node={node [sloped, allow upside down, auto=false]{$#1$}}}
  }
}
\makeatletter
\@namedef{subjclassname@2020}{%
  \textup{2020} Mathematics Subject Classification}
\makeatother
\begin{document}
\maketitle
\begin{abstract} 
We use the theory of logarithmic line bundles to construct compactifications of spaces of roots of a line bundle on a family of curves, generalising work of a number of authors. This runs via a study of the torsion in the tropical and logarithmic jacobians (recently constructed by Molcho and Wise). Our moduli space carries a `double ramification cycle' measuring the locus where the given root is isomorphic to the trivial bundle, and we give a tautological formula for this class in the language of piecewise polynomial functions (as recently developed by Molcho-Pandharipande-Schmitt and Holmes-Schwarz). 
\end{abstract}

%
%
%

\tableofcontents
\newcommand{\Spec}{\on{Spec}}
\newcommand{\Mtildes}{ \widetilde{\ca M}^\Sigma}
\newcommand{\sch}[1]{\textcolor{blue}{#1}}
\newcommand{\Mbar}{\overline{\ca M}}
\newcommand{\MD}{\ca M^\blacklozenge}
\newcommand{\Md}{\ca M^\lozenge}
\newcommand{\DRL}{\operatorname{DRL}}
\newcommand{\DR}{\sf{DR}}
\newcommand{\DRC}{\operatorname{DRC}}
\newcommand{\isom}{\stackrel{\sim}{\longrightarrow}}
\newcommand{\Ann}[1]{\on{Ann}(#1)}
\newcommand{\fm}{\mathfrak m}
\newcommand{\Mdk}{\Mbar^{\m, 1/\k}}
\newcommand{\field}{K}
\newcommand{\Mdm}{\Mbar^\m}
\newcommand{\m}{{\bd m}}
\newcommand{\cat}[1]{\bd{#1}}
\newcommand{\M}{\mathsf{M}}
\newcommand{\ghost}{\overline{\mathsf{M}}}
\newcommand{\gp}{\mathsf{gp}}
\newcommand{\op}{\mathsf{op}}
\newcommand{\trop}{\mathsf{trop}}
\newcommand{\LOG}{\mathsf{log}}
\newcommand{\fib}{\mathsf{cat}}
\newcommand{\et}{\mathsf{\acute{e}t}}
\renewcommand{\sf}[1]{\mathsf{#1}}
\newcommand{\piR}{{\pi_R}}
\newcommand{\jrtC}{{\sqrt[j]{X}}}
\newcommand{\LOGPIC}{\mathfrak{LogPic}}
\newcommand{\Pic}{\on{Pic}}
\newcommand{\pic}{\on{Pic}}
\newcommand{\LogPic}{\on{LogPic}}
\newcommand{\TroPic}{\on{TroPic}}
\newcommand{\logPic}{\on{logPic}}
\newcommand{\troPic}{\on{troPic}}
\newcommand{\LogSch}{\cat{LogSch}}
\newcommand{\CFGlog}[1]{\Phi(#1)}
\newcommand{\minCFG}[1]{\myvec{#1}} 
\newcommand{\roots}[3][r]{#2(#3^{\frac{1}{#1}})}
\section{Introduction}
For a family of smooth proper curves $X/S$ and a positive integer $r$, the $r$-torsion of the relative jacobian, here denoted $J[r]$, is a finite flat cover of $S$, playing a central role everywhere from Galois representations in number theory to the double ramification cycle in Gromov-Witten theory. If $X$ is equipped with a line bundle $\ca L$ of relative degree divisible by $r$, then the $r$-th roots of $\ca L$, denoted\footnote{So in particular $\roots[r]{X}{\ca O} = J[r]$.} $\roots[r]{X}{\ca L}$, are a torsor under $J[r]$, and play a comparably important role, going under the name of \emph{(higher) spin structures} when $\ca L$ is a power of the relative canonical bundle of $X/S$. 
Our goal in this paper is to extend this picture to the case of \emph{prestable} curves; proper flat curves whose geometric fibres are reduced and connected and have at-worst nodal singularities. Such curves appear everywhere in Gromov-Witten theory, and restricting to curves with such mild singularities is for some applications harmless, by the semistable reduction theorem. 

It is not possible to extend $J[r]$ to a finite flat group scheme over $S$ in the prestable case; already when $r=2$ and $X/S$ is the universal elliptic curve over the $j$-line in characteristic zero, one can show that $J[2]$ does not admit any unramified extension over $S$. However, we will see that this ramification is the only obstruction. Our first main theorem states that, after passing to a certain root stack $\tilde S \to S$ which `eats up all the ramification', all of the good properties enjoyed by $J[r]$ and by $\roots[r]{X}{\ca L}$ extend to the prestable case:
\begin{theorem}\label{thm:intro_1}
There is a functorial extension of $J[r]$ to a finite flat group scheme over $\tilde S$, and functorial extensions of the $\roots[r]{X}{\ca L}$ to finite flat $J[r]$-torsors over $\tilde S$. 
\end{theorem}
The precise definition of the root stack $\tilde S \to S$ is little delicate, and can be found in \ref{sec:root_stack}; for now we mention that
\begin{enumerate}
\item
$\tilde S \to S$ is proper, flat and quasi-finite;
\item $\tilde S \to S$ is an isomorphism over the open locus in $S$ of smooth curves, and even the larger locus of compact-type curves. 
\end{enumerate}
\Cref{thm:intro_1} was known due to work of Chiodo \cite{Chiodo2008Stable-twisted-} under the restrictions that
\begin{enumerate}
\item
We work over a field in which $r$ is invertible;
\item We take $X/S$ to be the universal stable curve over a Deligne-Mumford-Knudsen moduli space $\Mbar_{g,n}$ of stable curves;
\item The line bundle $\ca L$ is a power of the relative dualising sheaf $\omega_{X/S}$;
\end{enumerate}
Moreover, we will see in \ref{sec:examples} that the root stacks Chiodo constructs are significantly more ramified than our $\tilde S \to S$ . 
The construction of Chiodo runs via line bundles on twisted curves (prestable curves equipped with some stacky structure at their nodes), following ideas also explored by Jarvis and Abramovich \cite{jarvis1998torsion,abramovich2003moduli}. In the present paper these are replaced by logarithmic curves in the sense of Kato \cite{Kato1989Logarithmic-str}. Our functorial extensions of $J[r]$ and of $\roots[r]{X}{\ca L}$ are described in terms of log line bundles on log curves, in the sense of Molcho and Wise \cite{Molcho2018The-logarithmic}; the present paper can be seen as a careful study of the torsion in their logarithmic and tropical jacobians. We can also state a logarithmic version of our main theorem, where no root stack $\tilde S \to S$ is required: 
\begin{theorem}\label{thm:intro_2}
There is a functorial extension of $J[r]$ to a finite flat logarithmic group scheme over $S$, and functorial extensions of the $\roots[r]{X}{\ca L}$ to finite flat logarithmic $J[r]$-torsors over $S$. 
\end{theorem}
The fact that the root stack $\tilde S \to S$ is necessary for the non-logarithmic statement comes from the fact that the underlying scheme of a logarithmic group scheme need not carry a group structure, see \ref{sec:log_version} for a more extended discussion of this point. 

\begin{remark}
Work of Caporaso, Casagrande, and Cornalba \cite{CaporasoModuli-of-roots} constructed compactifications of moduli spaces of roots under less restrictive conditions than Chiodo and Jarvis (for example, allowing arbitrary prestable families $X/S$ as we do), but they do not give their compactifications group or torsor structures. 
\end{remark}

The universal curve over the stack $\roots[r]{X}{\ca L}$ carries (locally) a logarithmic line bundle $\ca F$ such that $\ca F^{\otimes r} \cong \ca L$. The locus in $\roots[r]{X}{\ca L}$ over which $\ca F$ is fibrewise (logarithmically) trivial carries a natural virtual fundamental class, the \emph{$r$-spin double ramification cycle $\on{DR}(\ca L^{1/r})$ of $\ca L$}. Our final result is an expression for $\on{DR}(\ca L^{1/r})$ in terms of piecewise polynomial functions on the tropicalisation of $\roots[r]{X}{\ca L}$, and may be viewed as a continuation of the story told by \cite{Janda2016Double-ramifica,Janda2018Double-ramifica,Bae2020Pixtons-formula,Molcho2021The-Hodge-bundl,Holmes2021Logarithmic-int}. The full expression is lengthy to write out (see \ref{sec:DR_appendix}), but as an immediate corollary we obtain that $\on{DR}(\ca L^{1/r})$ is `tautological' in the following sense. 
\begin{theorem}
The class $\on{DR}(\ca L^{1/r})$ in $\on{CH}(\roots[r]{X}{\ca L})$ lies in the subring generated by the images of piecewise polynomials on the tropicalisation of $\roots[r]{X}{\ca L}$, together with the class $\eta_\ca L = \pi_*(c_1(\ca L)^2)$. 
\end{theorem}

\subsection{Strategy of proof}
We first prove \ref{thm:intro_2}; we construct our compactifications of $J[r]$ and of $\roots[r]{X}{\ca L}$ from the torsion in the logarithmic jacobian. The logarithmic jacobian is not a log scheme, but is rather a `logarithmic space'; a sheaf on the category of log schemes admitting a log \'etale cover by a log scheme. Proving that the torsion is in fact a logarithmic scheme goes via a careful analysis of the minimal objects (in the sense of Gillam \cite{Gillam2012Logarithmic-sta}). 

To derive \ref{thm:intro_1} we first need to build the root stack $\tilde S \to S$; this is done by an explicit construction in \ref{sec:root_stack}. We then prove that $J[r]$ and $\roots[r]{X}{\ca L}$ become strict after base-change to $\tilde S$, from which the theorem easily follows. 

Finally, to give a formula for $\on{DR}(\ca L^{1/r})$ we start by giving an explicit description of the cones of $\roots[r]{X}{\ca L}$ on which we will write our piecewise-polynomials. We then realise the log line bundle $\ca F$ as a line bundle on a suitable subdivision of the universal curve over $\roots[r]{X}{\ca L}$, allowing us to apply the main theorem of \cite{Bae2020Pixtons-formula} to derive our formula.


\subsection{Acknowledgements}
D.H. was supported by NWO grants VI.Vidi.193.006 and 613.009.103. G.O. was supported by the EPFL Chair of Arithmetic Geometry (ARG). We are grateful to Georgios Politopoulos for pointing out some small errors in the formulae in \ref{sec:DR_appendix}. 

\subsubsection*{Attribution}
G.O. wishes to acknowledge that \cref{sec:DR_appendix} was written by D.H. 

\section{Background}
\subsection{Log structures}

We write $\M^\times$ for the units of a (commutative) monoid $\M$, and $\M^\gp$ for its groupification. 
Following \cite{Kato1989Logarithmic-str}, a \emph{log structure} on an algebraic space $X$ consists of an \'etale sheaf of monoids $\M_X$ on $X$ together with a map of sheaves $\alpha\colon \M_X \to \ca O_X$, such that $\alpha\colon \alpha^{-1}\ca O_X^\times \to \ca O_X^\times$ is an isomorphism; we refer to \cite{Kato1989Logarithmic-str} and \cite{Ogus2018Lectures-on-log} for further details. The \emph{ghost sheaf} is $\ghost_X \coloneqq \M_X/\M_X^\times$. We write $\ul X$ for the underlying algebraic space of a log algebraic space $X$, and $X = (\ul X, \M_X)$. We work throughout with fine saturated (fs) log structures. The category of (fs) log schemes is denoted $\cat{LogSch}$, and comes with a forgetful functor to the category $\cat{Sch}$ of schemes. The wide subcategory of $\cat{LogSch}$ with only strict morphisms is denoted $\cat{Log}$; the corresponding functor $\cat{Log} \to \cat{Sch}$ is a groupoid fibration (see \cite{Olsson2001Log-algebraic-s}). 

\subsection{Minimal log structures}
A log scheme $X$ induces a Yoneda functor $h_X\colon \cat{LogSch}^\op \to \cat{Set}$, but recovering the functor $h_{\ul X} \colon \cat{Sch}^\op \to \cat{Set}$ from $h_X$ is a somewhat delicate problem, which has been completely solved by Gillam \cite{Gillam2012Logarithmic-sta} in a more general context. This will play a central role in what follows, so we briefly summarise Gillam's results, as well as adding a couple of lemmas. 

We write $\cat{CFG}_{\cat{LogSch}}$ for the 2-category of CFGs (categories fibred in groupoids) over $\cat{LogSch}$. If $X \to \cat{Sch}$ is a CFG then a \emph{log structure} on $X$ is a factorisation $X \to \cat{Log}\to \cat{Sch}$. 
We write $\cat{LogCFG}_{\cat{Sch}}$ for the 2-category whose objects are CFGs over $\cat{Sch}$ equipped with a log structure. 

 We write 
\begin{equation}
\Phi\colon \cat{LogCFG}_{\cat{Sch}} \to \cat{CFG}_{\cat{LogSch}}
\end{equation}
for the natural 2-functor; this is analogous to taking a log scheme $X$ and looking at the resulting Yoneda functor $h_X\colon \cat{LogSch}^\op \to \cat{Set}$. 
Gillam \cite{Gillam2012Logarithmic-sta} proves that this 2-functor is fully faithful, and gives a precise characterisation of the essential image. Starting with an object $X$ of $\cat{CFG}_{\cat{LogSch}}$, he defines a subset of the objects of $X$ which he calls the \emph{minimal} objects; we write $\minCFG{X}$ for the full subcategory of $X$ whose objects are minimal. Gillam shows that the CFG $X$ lies in the essential image of $\Phi$ if and only if $X$ has `enough' minimal objects (in a precise sense) and that, in that case, the fibred category $\minCFG{X} \to \cat{CFG}_{\cat{LogSch}} \to \cat{CFG}_{\cat{Sch}}$ (where the latter morphism forgets the log structure) has the property that $\Phi(\minCFG{X}) = X$. 
%
%
%
%
%
%
%
%
\begin{example}\label{eg:P1modGm}
Let $X$ be the category fibred in setoids over $\cat{LogSch}$ whose objects over a log scheme $T$ are elements $m \in \ghost_T^\gp$ such that either $m$ or $-m$ lies in $\ghost_T$. This lies in the essential image of $\Phi$. More precisely, equip $\bb P^1$ with its standard toric log structure, and write $[\bb P^1/\bb G_m]$ for the quotient by the scaling action. This is a logarithmic algebraic stack, in particular an object of $\cat{LogCFG}_{\cat{Sch}}$, and $\Phi([\bb P^1/\bb G_m]) = X$. Here we see how taking the preimage under $\Phi$ can create extra inertia; essentially the log structure rigidifies away some of that inertia. This kind of phenomenon will play an important role in what follows, especially in \ref{sec:underlying_stack_of_LogPic_torsion}. 
\end{example}
Let $\phi\colon X \to Y$ be a 1-morphism in $\cat{CFG}_\cat{LogSch}$, and suppose that $\phi$ is representable by strict objects in $\cat{LogCFG}_{\cat{Sch}}$; by this we mean that for every $T$ in $\cat{LogCFG}_{\cat{Sch}}$ and 1-morphism $\Phi(T) \to Y$, the fibre product $X \times_Y \Phi(T)$ lies in the essential image of $\Phi$, and moreover that the morphism $X \times_Y \Phi(T) \to \Phi(T)$ is strict. Suppose in addition that the relative inertia of $\phi$ is trivial, so that $X \times_Y \Phi(T)$ is in fact a category fibred in setoids over $\cat{Sch}$. 
%
%
%
%
%
%
One deduces quite easily from Gillam's machinery the next two lemmas. 
\begin{lemma}\label{lem:strict_preserves_minimal_objects}
The minimal objects of $X$ are exactly those objects which map to minimal objects of $Y$.
\end{lemma}
Note that strictness is needed for both inclusions of minimal objects, even when $X$ and $Y$ are taken to be log schemes. To see this consider the (non-strict) morphism 
\begin{equation}
\phi\colon \Spec \bb Z[\bb N^2] \to \Spec \bb Z[\bb N] ; (1,1) \mapsfrom 1. 
\end{equation}
Here the identity on $\Spec \bb Z[\bb N^2]$ is minimal but does not map to a minimal object, and the object
\begin{equation}
\Spec \bb Z[\bb N] \to \Spec \bb Z[\bb N^2] ; a \mapsfrom (a,b)
\end{equation}
is not minimal but is mapped to a minimal object. 
\begin{lemma}\label{strict_over_minCFG_has_minCFG}
Suppose in addition that $Y$ lies in the essential image of $\Phi$. Then $X$ lies in the essential image of $\Phi$. 
\end{lemma}
\subsection{Root stacks}

Logarithmic geometry gives a powerful perspective on root stacks, as we now describe. The relation to classical root stacks is explained in \ref{eg:root_stack}. 
For our input data, let $X$ be a log scheme, and let $j\colon \ghost_X \to R$ a morphism of sheaves of monoids on $X$ such that the saturation of the image of $j$ in $R$ is equal to $R$. 
\begin{definition}\label{def:root_stack_log}
The \emph{$j$th root of $X$} is the full subcategory $\sqrt[j]{X}$ of $\CFGlog X$ consisting of those $t\colon T\to X$ such that the map $t^{-1}\ghost_X\to \ghost_T$ factors via $t^{-1}j\colon t^{-1}\ghost_X\to t^{-1}R$. 
\end{definition}
The factorization in \ref{def:root_stack_log} is automatically unique, as a consequence of the following lemma.
\begin{lemma}\label{lemma:unique_fact}
Let $X$, $j\colon \ghost_X\to R$ as above. Let $f,g\colon R \to N$ be two morphisms to a sheaf of sharp and saturated monoids. If $j\circ f=j\circ g$ then $f=g$.  
\end{lemma}
\begin{proof}
Let $f,g\colon R\to N$ be two morphisms to a sharp $N$ such that $f\circ j=g\circ j$. Let $r$ be a section of $R$; then after possibly localizing, there exists an integer $d>0$ such that $dr$ is in the image of $j$. Hence $f(dr)=g(dr)$, i.e. $d(f(r)-g(r))=0$ in $N^{gp}$. As $N$ is saturated, $f(r)-g(r)$ is in $N$; as $N$ is integral, $d(f(r)-g(r))=0$ in $N$; and finally because $N$ is sharp $f(r)=g(r)$ in $N$. 
\end{proof}
\begin{lemma}\label{def:root_stack_minimal}
The subcategory $\minCFG{\sqrt[j]{X}}$ of minimal objects is equal to the subcategory of $\sqrt[j]{X}$ consisting of those $t\colon T\to X$ for which the factorization $\ghost_T\to t^*R$ is an isomorphism. 
\end{lemma}
Recall that morphisms in $\minCFG{\sqrt[j]{X}}$ are induced by morphisms of $X$-log schemes $T\to T'$ that are strict. We can make $\minCFG{\sqrt[j]{X}}$ into a category fibered in groupoids over $\cat{Sch}/\ul{X}$ by sending $T\to X$ to the underlying morphism of schemes.

In \cite[Def. 4.16]{Borne2012Parabolic-sheav} Borne and Vistoli define similarly a stack of roots $X_{R|\ghost_X}$ via the machinery of Deligne-Faltings log structures. 
Comparing the definitions one easily verifies 
 
\begin{lemma}\label{lem:BV_equivalence}
There is natural equivalence $\minCFG{\sqrt[j]{X}} \to X_{R|\ghost_X}$ of CFG over $\cat{Sch}/\ul{X}$. 
\end{lemma} 
\begin{definition} [{\cite[4.1, 4.3]{Borne2012Parabolic-sheav}}] \label{def:system_of_denominators}
A \textit{Kummer homomorphism} $f\colon A\to B$ of finitely generated monoids is an injective morphism such that for every $b\in B$ there exists a positive integer $d$ such that $db$ is in $A$.
Let $X$ be a scheme. A \textit{system of denominators} is a morphism of coherent sheaves of monoids $A\to B$ on $X_{et}$ such that the induced morphisms at geometric stalks are Kummer homomorphisms.
\end{definition}
\begin{lemma}\label{lemma_rootstack_representable}
Let $X$ be a log scheme and $j\colon \ghost_X\to R$ be a system of denominators. Then $\minCFG{\sqrt[j]{X}} \to X$ is relatively representable by proper, quasi-finite, finitely presented algebraic stacks. If moreover for every geometric point $x\colon \Spec \Omega \to X$ the order of $R^{gp}/\ghost_X^{gp}$ is prime to the characteristic of $\Omega$, then $\minCFG{\sqrt[j]{X}} \to X$ is representable by Deligne-Mumford stacks. 
\end{lemma}
\begin{proof}
This is proposition 4.19 of \cite{Borne2012Parabolic-sheav}, which applies by \ref{lem:BV_equivalence}. 
\end{proof}

\begin{example}\label{eg:root_stack}
Let $X$ be a regular scheme and $\iota\colon D\hookrightarrow X$ an irreducible closed subscheme of codimension $1$. Given an integer $r>0$ invertible on $X$, there is a classical notion of root stack of order $r$ along $D$, denoted $\sqrt[r]{X,D} \to X$; it is the stack over $X$ whose points over $t\colon T\to X$ are triples of a line bundle $\ca L$ on $T$, a section $l$ of $\ca L$, and an isomorphism $\ca L^{\otimes r}\to t^*\ca O(D)$ sending $l^{\otimes r}$ to $1$.
 The map $\sqrt[r]{X,D} \to X$ is obtained, locally on $X$, by taking a finite flat cover $Y\to X$ obtained by extracting a $r$-root of a function cutting out $D$; and then taking the stack quotient $[Y/\mu_r]$ by the natural action of $\mu_r$ on $Y$.
 
If we endow $X$ with its divisorial log structure $\mathsf M_X$ induced by $D$ then the ghost sheaf $\ghost_X$ is naturally identified with $\iota_*\bb N$, and in particular has a unique generating section $e$. Writing $j\colon \ghost_X\to \ghost_X$ for the multiplication by $r$, there is a natural isomorphism $\minCFG{\sqrt[j]{X}} \to \sqrt[r]{X,D}$ sending an object $t\colon T\to \minCFG{\sqrt[j]{X}}$ to a triple $(\mathcal L,l, \phi)$ constructed as follows: the image of $e$ in $\ghost_T$ is uniquely divisible by $r$, say $re' = e$. The preimage of $e'$ under $\mathsf M_T\to \ghost_T$ is a $\mathcal O^{\times}_T$-torsor, denoted $\mathcal O^{\times}_T(-e)$. The associated invertible sheaf $\mathcal O_T(-e)$ comes with a map to $\mathcal O_T$ induced by the log structure. Then we take $\mathcal L$ to be its dual $\mathcal O_T(e)$ and $l$ to be its distinguished global section. The isomorphism $\phi$ is induced by the equality $re' = e$. 
%
%
%
\end{example}
\subsection{Log curves}\label{sec:log_curves}
\begin{definition}[\cite{Kato2000Log-smooth-defo}]
Let $S$ be a log scheme. A morphism $\pi\colon X \to S$ of log algebraic spaces is a \emph{log curve} if it is proper, log smooth, integral, and all geometric fibres are reduced, connected and of pure dimension 1. 
\end{definition}
The underlying scheme of a log curve is a prestable curve. Unless otherwise noted, all log curves we consider will be \emph{vertical}; in other words, the corresponding prestable curves are unmarked. 

Let $k$ be an algebraically closed field and $X /(\Spec k,\M)$ a log curve. From the description of the local structure of log curves at nodes (see \cite{Kato2000Log-smooth-defo}), at every node $p\in X$ there is a canonical non-zero element $\delta_p \in \ghost$ and isomorphism 
\begin{equation}\ghost_{X,p}\cong \frac{\ghost\oplus \bb N\langle e_1\rangle \oplus \bb N\langle e_2 \rangle}{e_1+e_2=\delta_p}. 
\end{equation}
Writing $\frak X$ for the dual graph of $X$, we label each edge $e_p$ of $\frak X$ by the associated element $\delta_p\in \ghost$. We think of these labels as the \emph{lengths} of the edges; in the special case where $\ghost = \bb N$ this fits with the intuitive notion of a graph with positive integer edge lengths. The metric graph $\frak X$ is the \emph{tropicalisation} of $X$. 


Let $X/S$ be a log curve. To every geometric fibre $X_{s}$ one can associate the dual graph $\frak X_{s}$ of the log curve; and in turn to this graph is associated the free finitely generated abelian group $H_1(\frak X_{s},\bb Z)$, whose rank is the first Betti number of $\frak X_{s}$. In \cite[Definition 3.29]{Holmes2020Models-of-Jacob} the authors introduce a locally quasi-finite \'etale $S$-algebraic space, the \emph{tropical homology of $X/S$}, denoted $\mathcal H_{1,X/S}$, whose geometric fibres are naturally identified with the groups $H_1(\frak X_{s},\bb Z)$. Note that $\mathcal H_{1,X/S}$ is generally very far from being separated over $S$, as it maps isomorphically to $S$ over the locus of compact-type curves, and has infinite fibre over any point outside that locus.

\subsection{The logarithmic and tropical jacobians}
In this section we briefly recall the key definitions from \cite{Molcho2018The-logarithmic}, to which we refer for details and proofs. We begin by introducing two sheaves of abelian groups on $\cat{LogSch}/S$ for the strict \'etale topology:
\begin{align*}
\bb G_m^{trop}\colon T &\mapsto \ghost_T^{gp}(T) \\
\bb G_m^{log}\colon  T & \mapsto \M_T^{gp}(T), 
\end{align*}
referred to as the \emph{tropical} and \emph{logarithmic} multiplicative groups respectively. 

\subsubsection{The intersection pairing}
The tropical homology $\mathcal H_{1,X/S}$ of \ref{sec:log_curves} is equipped with a natural \emph{intersection pairing}
\begin{equation}
\cap \colon \mathcal H_{1,X/S}\times \mathcal H_{1,X/S} \to \bb G_m^{trop}, 
\end{equation}
which over a geometric point $s$ of $S$ sends two directed cycles $\gamma$, $\gamma' \in  \mathcal H_{1,X/S}$ to the sum 
\begin{equation}
\gamma \cap \gamma' \coloneqq \sum_{\ul e \in \gamma} \gamma'(\ul e) \delta_e
\end{equation}
over directed edges $\ul e$ in $\gamma$ where (writing $-\ul e$ for the same edge with opposite orientation) we set
\begin{equation}
\gamma'(\ul e) = 
\begin{cases}
1 & \text{if } \ul e \in \gamma' \\
-1 & \text{if } -\ul e \in \gamma' \\
0 & \text{else. }\\
\end{cases}
\end{equation}

\subsubsection{The bounded monodromy condition}
Given two elements $a,b$ of a monoid $M$, we say that $a$ is \emph{bounded by $b$} if there exist integers $m$ and $n$ such that both $a - mb$ and $nb - a$ lie in $M$

Given a log curve $X$ over $(\on{Spec} k, \ghost)$ with $k$ an algebraically closed field $k$, and given a map $\alpha\colon H_1(\frak X) \to \ghost^\gp$, we say that $\alpha$ has \emph{bounded monodromy} around a cycle $\gamma \in H_1(\frak X)$ if $\alpha(\gamma)$ is bounded by the length of $\gamma$ (i.e. by the sum of the lengths of the edges appearing in $\gamma$). We say $\alpha$ has \emph{bounded monodromy} if it has so around $\gamma$ for all $\gamma \in H_1(\frak X)$. 

If $X/S$ is a log curve, we write 
\begin{equation}
\sheafhom(\ca H_{1, X/s}, \bb G_m^\trop)^\dagger \subseteq \sheafhom(\ca H_{1, X/s}, \bb G_m^\trop)
\end{equation}
for the subsheaf consisting of those homomorphisms which have bounded monodromy on each geometric fibre. 

This condition of bounded monodromy plays a central role in the construction of the logarithmic and tropical jacobians in \cite{Molcho2018The-logarithmic}. However, for us it will not be so prominent, as we are mainly intersected in torsion points, for which the condition is automatic.

\subsubsection{The logarithmic and tropical jacobians}

The intersection pairing furnishes a natural injective map 
\begin{equation}
\ca H_{1, X/S} \to \sheafhom(\ca H_{1, X/s}, \bb G_m^\trop)
\end{equation}
which on geometric fibres sends a cycle $\gamma$ to the map $\gamma' \mapsto \gamma' \cap \gamma$. One checks easily that the image of $\ca H_{1, X/s}$ has bounded monodromy, inducing an injection 
\begin{equation}\label{eg:cycles_to_bm_maps}
\ca H_{1, X/S} \to \sheafhom(\ca H_{1, X/s}, \bb G_m^\trop)^\dagger. 
\end{equation}
\begin{definition}\label{def:tro_jac}
The \emph{tropical jacobian} $\TroPic^0_{X/S}$ is the cokernel of \ref{eg:cycles_to_bm_maps}, as a functor from log schemes over $S$ to abelian groups. 
\end{definition}
Molcho and Wise also define a larger object, the \emph{tropical Picard space} $\TroPic_{X/S}$, of which $\TroPic^0_{X/S}$ is the degree-zero part. But the details of the definition are not so important to us, as our technical results concern torsion elements, which always have degree zero. 

If $X$ is a log scheme, a \emph{log line bundle} on $X$ is a $\bb G_m^\LOG$-torsor for the strict \'etale topology. If $X/S$ is a log curve and $\ca L$ on $X$ a log line bundle, the natural map $\bb G_m^\LOG \to \bb G_m09^\trop$ induces a $\bb G_m^\trop$-torsor $\bar {\ca L}$ on $X$. If $S$ is a geometric point then we identify $\bar {\ca L}$ with an element of $H^1(\frak X, \ghost^\gp) = Hom(H_1(\frak X), \ghost^\gp)$, and we say that $\ca L$ has \emph{bounded monodromy} if the corresponding map $H_1(\frak X)\to \ghost^\gp$ does so. For general $S$ we say $\ca L$ has bounded monodromy if it does so over every geometric point of $S$. 

\begin{definition}
The \emph{log Picard space} of $X/S$, denoted $\LogPic_{X/S}$, is the strict-\'etale sheaf on $S$ of isomorphism classes of log line bundles on $X$ with bounded monodromy. 
\end{definition}
Logarithmic line bundles have an integer-valued degree, and we write $\LogPic^0_{X/S}$ for the degree-zero part. 

Analogously we write $\Pic_{X/S}$ for the relative Picard space of $X/S$, and $\Pic_{X/S}^{\ul 0}$ for the open subgroup consisting of those line bundles with \emph{multidegree 0}, i.e. degree 0 on each irreducible component of each geometric fibre of $X/S$. The natural map $\ca O_X^\times \to \M_X^\gp$ induces a map
\begin{equation}
\Pic_{X/S} \to \LogPic_{X/S}
\end{equation}
which preserves total degree. This induces an exact sequence
\begin{equation}\label{SES:pic_log_tro}
0 \to \Pic_{X/S}^{\ul 0} \to \LogPic_{X/S} \to \TroPic_{X/S} \to 0
\end{equation}
which restricts to an exact sequence 
\begin{equation}
0 \to \Pic_{X/S}^{\ul 0} \to \LogPic^0_{X/S} \to \TroPic^0_{X/S} \to 0. 
\end{equation}

\section{Multiplication by $r$ on $\TroPic$ and $\LogPic$}
Let $X/S$ be a log curve and $r$ a positive integer. In this section we prove certain basic structural results about the `multiplication by $r$' map on the logarithmic and tropical jacobians of $X/S$. We begin by summarising these results. 
\begin{theorem}
The multiplication by $r$ map $\TroPic_{X/S} \to \TroPic_{X/S}$ is:
\begin{enumerate}
\item Relatively representable by algebraic stacks with log structure.
\item Proper, integral, log flat and quasi-finite.
\item Is an isomorphism in degree 0 over the open locus in $S$ over which $X$ is of compact type. 
\item Is log \'etale if $r$ is invertible on $S$. 
\end{enumerate}
\end{theorem}
\begin{theorem}
The multiplication by $r$ map $\LogPic_{X/S} \to \LogPic_{X/S}$ is representable by proper flat quasi-finite algebraic spaces with log structure. If the $r$-torsion $\TroPic_{X/S}[r] \to S$ is strict then the underlying map of algebraic spaces $\ul{\LogPic}_{X/S} \to \ul{\LogPic}_{X/S}$ is \'etale and respects the group structures. 
\end{theorem}
Proving these results will occupy the remainder of this section. Applications to compactifying spaces of roots of line bundles will be pursued in \ref{sec:compactifying_roots}. 
\subsection{The basic log structures ${\ghost}_{\gamma/r}$}
Fix a positive integer $r$. Let $X/S$ be a nuclear log curve in the sense of \cite[definition 3.39]{Holmes2020Models-of-Jacob} (essentially, $S$ is `small enough'), with base monoid $\ghost = \ghost_S(S)$, and let $\gamma \in H_1(\frak X)$. 
\begin{definition}\label{def:M_gamma_over_r}
We define ${\ghost}_{\gamma/r}$ to be the saturation of $\ghost$ inside
\begin{equation}
\mathsf N_\gamma \coloneqq \frac{\ghost^\gp \oplus H_1(\frak X)}{(\gamma \cdot \gamma', -r\gamma') : \gamma' \in H_1(\frak X))}, 
\end{equation}
and we define
\begin{equation}\label{eq:f_gamma_over_r_map}
f_{\gamma/r}\colon H_1(\frak X) \to {\ghost}_{\gamma/r}; \gamma' \mapsto (0, \gamma'). 
\end{equation}
\end{definition}
We verify that this definition depends only on the equivalence class of $\gamma$ in the quotient $H_1(\frak X)/rH_1(\frak X)$:
\begin{lemma}
Let $\gamma_1$, $\gamma_2 \in H_1(\frak X)$. 
\begin{enumerate}
\item
If $\gamma_1 - \gamma_2 \in rH_1(\frak X)$ then there is a unique isomorphism of monoids ${\ghost}_{\gamma_1/r} \to {\ghost}_{\gamma_2/r}$ which is the identity on $\ghost$, and it takes $f_{\gamma_1/r}$ to $f_{\gamma_2/r}$. 
\item
Suppose there exists an isomorphism of monoids ${\ghost}_{\gamma_1/r} \to {\ghost}_{\gamma_2/r}$ which is the identity on $\ghost$. Then for all $\gamma \in H_1(\frak X)$ we have $(\gamma_1 - \gamma_2)\cdot \gamma \in r\ghost_X$. 
\end{enumerate}
\end{lemma}
\begin{proof}
\begin{enumerate}
\item First we show the map exists. We write $\gamma_1 - \gamma_2 = r\gamma_0$, and define
\begin{equation}
\phi\colon \ghost^\gp \oplus H_1(\frak X)  \to \ghost^\gp \oplus H_1(\frak X); (m, \gamma) \mapsto (m + \gamma\cdot \gamma_0, \gamma). 
\end{equation}
This map evidently commutes with the inclusions of $\ghost$ and with the maps $f_{\gamma_i/r}$, and we compute
\begin{equation}
\phi(\gamma_1 \cdot \gamma, -r\gamma) = (\gamma_1 \cdot \gamma - r\gamma_r \cdot \gamma, -r\gamma) = (\gamma_2 \cdot \gamma, -r\gamma), 
\end{equation}
hence $\phi$ induces a map on the ${\ghost}_{\gamma_i/r}$. 
For the uniqueness of the map, it is enough to check this on the submonoids $r{\ghost}_{\gamma_i/r}$, but these submonoids are contained in $\ghost$, and we are assuming our map is the identity on $\ghost$. 
\item 
If such a map existed it would have to take $(rm, r\gamma)$ to $(rm + \gamma\cdot (\gamma_1 - \gamma_2), r\gamma)$, and the latter is divisible by $r$. 
\end{enumerate}
\end{proof}
\begin{lemma}\label{lem:inclusion_integral}
The inclusion $\ghost \to \ghost_{\gamma/r}$ is integral\footnote{See \cite[definition 4.6.2]{Ogus2018Lectures-on-log} for a definition of integral morphism of monoids.}. 
\end{lemma}
\begin{proof}
We  have to show the following: for every $m_1,m_2\in \ghost$ and $n_1,n_2\in \ghost_{\gamma/r}$ such that $m_1+n_1=m_2+n_2$ in $\ghost_{\gamma/r}$, there exists an $n'\in \ghost_{\gamma/r}$ and $m'_1,m'_2\in \ghost$ such that $n_i=n'+m'_i$ and $m_1+m'_1=m_2+m'_2$. 
Fix a basis $\gamma_1,\ldots,\gamma_g$ for $H_1(\frak X)$. From the exact sequence
\begin{equation}
0 \to \ghost^\gp \to \mathsf N_\gamma \to H_1(\frak X)/rH_1(\frak X)
\end{equation}
we deduce that for every element $n\in \mathsf N_\gamma$ there exists a unique element $\alpha(n)\in \ghost^\gp$ and integers $0\leq a_1(n),\ldots,a_g(n) < r$ such that $n=\alpha(n)+\sum_1^g a_j(n)\gamma_j$. Moreover if $n \in \ghost_{\gamma/r} \sub \mathsf N_\gamma$ then $rn \in \ghost$ so $\alpha(rn) = rn$, hence $\alpha(n) \in \ghost$ by saturation of $\ghost$. We write $\alpha\colon \ghost_{\gamma/r}\to \ghost$ and $a_i\colon \ghost_{\gamma/r}\to \{0,\ldots,r-1\}$ for the functions so constructed. 
Now let $m_1,m_2,n_1,n_2$ be as above. We have $a_j(n_1)=a_j(n_1+m_1)=a_j(n_2+m_2)=a_j(n_2)$ for all $0\leq j <r$. Let $n'=\sum a_j(n_1)\gamma_j$ and for $i=1,2$ let $m'_i=\alpha(n_i)$ so that we have $n_i=n'+m'_i$ which fulfils the first condition we needed to check. This implies moreover that $m_1+m'_1+n'=m_2+m'_2+n'$, hence $m_1+m'_1=m_2+m'_2$ by integrality of $\ghost$. 
\end{proof}
\subsection{Realising $\TroPic[r]$ as a root stack}
We write $\mathcal H_1:=\mathcal H_{1,X/S}$ for the tropical homology of $X/S$. From the natural exact sequence
\begin{equation}
0\to \ca H_1 \to \sheafhom(\ca H_1,\bb G_m^{trop})^{\dagger} \to \TroPic^0 \to 0
\end{equation}
of \ref{def:tro_jac} and the fact that the middle term has no torsion,
we obtain an exact sequence
\begin{equation}
\label{eqn:tropic_r} 0\to \TroPic[r] \to \frac{\ca H_1}{r\ca H_1} \to \frac{\sheafhom(\ca H_1,\bb G_m^{trop})^{\dagger}}{r \sheafhom(\ca H_1,\bb G_m^{trop})^{\dagger}}. 
\end{equation}
Denote by $\ghost$ the ghost sheaf of the log algebraic space $\mathcal H_1/r\mathcal H_1$. The tropical curve $\mathfrak X\times_S \mathcal H_1/r\mathcal H_1$ carries a tautological cycle $\gamma$ well-defined up to multiples of $r$, which determines a system of denominators $\ghost \to \ghost_{\gamma/r}$ on $\mathcal H_1/r\mathcal H_1$ (see \ref{def:system_of_denominators}). 
\begin{lemma}\label{lemma_minimal_monoid_integral}
The inclusion $\TroPic_{X/S}[r] \hookrightarrow \mathcal H_1/r\mathcal H_1$ is the root stack $\sqrt[j]{\ca H_1/r\ca H_1}$ with respect to the monoid extension $j\colon \ghost \to \ghost_{\gamma/r}$. 
\end{lemma}
The reader concerned about how a root stack can be an inclusion may be reassured by recalling from \ref{sec:root_stack} that root stacks are log monomorphisms. 
\begin{proof}
From the exact sequence \ref{eqn:tropic_r} we see that for every $T\to S$ \[\TroPic[r](T)=\{\gamma\in \mathcal H_1/r\mathcal H_1(T) \mbox{ such that }\cap \gamma \mbox{ is divisible by } r\}. \]
However $\cap \gamma$ is divisible by $r$ if and only if $\ghost_S\to \ghost_{S,\gamma/r}$ is an isomorphism. The diagram
\begin{center}
\begin{tikzcd}
\ghost \ar[r] \ar[d] & \ghost_S \ar[d,"\cong"] \\
\ghost_{\gamma/r} \ar[r] & \ghost_{S,\gamma/r}
\end{tikzcd}
\end{center}
shows that this is in turn equivalent to the map of monoids $\ghost\to \ghost_S$ factoring via $\ghost_{\gamma/r}$. So we find that  
\[\TroPic[r](T)=\{\gamma\colon T\to \mathcal H_1/r\mathcal H_1 \mbox{ such that } \ghost \to \ghost_S \mbox{ factors via }\ghost_{\gamma/r}\}\]
which is exactly the desired root stack. 
\end{proof}
\begin{theorem}\label{thm:tro_pic_tors_properties}
The structure map $\minCFG{\TroPic[r]} \to S$ is:
\begin{enumerate}
\item
Relatively representable by algebraic stacks with log structure;
\item Finitely presented, log flat, flat and quasi-finite. 
\item If log \'etale If $r$ is invertible on $S$;
\item Is an isomorphism over the open locus in $S$ over which $X$ is of compact type. 
\end{enumerate}
\end{theorem}
\begin{proof}
The map is the composition of the strict \'etale morphism $\ca H_1/r\ca H_1\to S$  with the root stack $\minCFG{\TroPic[r]}=\minCFG{\sqrt[j]{\ca H_1/r\ca H_1}}\to \ca H_1/r\ca H_1$. All properties of parts (1) and (2) follow from \ref{lemma_rootstack_representable}, except flatness.
To prove flatness, we refer to \cite{Borne2012Parabolic-sheav}, paragraph above Proposition 4.13: let $P\to Q$ be an \'etale-local chart for the system of deniminators $j\colon \ghost\to \ghost_{\gamma/r}$ of \ref{lemma_minimal_monoid_integral}. The tori $\widehat{P}=Hom(P^{gp},\bb G_m)$ and $\widehat{Q}=Hom(Q^{gp},\bb G_m)$ act naturally on $\Spec \bb Z[P]$ and $\Spec \bb Z[Q]$ respectively and the actions make the map $\Spec \bb Z[Q]\to \Spec \bb Z[P]$ equivariant. The log structure on $\ca H_1/r\ca H_1$ induces a morphism from $\ca H_1/r\ca H_1$ to the quotient stack $[\Spec \bb Z[P]/\widehat{P}]$. Now \cite[Proposition 4.13]{Borne2012Parabolic-sheav} tells us that $\minCFG{\TroPic[r]}$ is the fibre product $\ca H_1/r\ca H_1\times_{[\Spec \bb Z[P]/\widehat{P}]}[\Spec \bb Z[Q]/\widehat{Q}]$. To prove flatness of $\minCFG{\TroPic[r]} \to \ca H_1/r\ca H_1$ it suffices therefore to prove flatness of $\Spec \bb Z[Q]\to \Spec \bb Z[P]$. This follows from \cite[Prop. 4.6.7]{Ogus2018Lectures-on-log} and \ref{lem:inclusion_integral}. 
Part (3) is proven in the same way as flatness, observing that the map $P^{gp}\to Q^{gp}$ has cokernel of order invertible on $S$ and $\Spec \bb Z[Q]\to \Spec \bb Z[P]$ is therefore log \'etale.
Finally part (4) follows from the fact that $\TroPic^0_{X/S}$ vanishes on the locus where $\ca H_1=0$, which is exactly the locus where $X/S$ is of compact-type.
\end{proof}
\subsection{On the underlying algebraic stack of the torsion of LogPic}\label{sec:underlying_stack_of_LogPic_torsion}

\begin{lemma}
The minimal objects of $\LogPic_{X/S}$ are exactly the objects which map to minimal objects of $\TroPic_{X/S}$, and $\LogPic_{X/S}$ is representable by an algebraic stack with log structure. 
\end{lemma}
\begin{proof}
Immediate from \ref{lem:strict_preserves_minimal_objects,strict_over_minCFG_has_minCFG}. 
\end{proof}

Let $T / S$ be a log scheme. Perhaps replacing $T$ by a strict \'etale cover, we may assume that $C_T/T$ admits a section and hence that any map $T \to \LogPic_{C/S}$ can be represented by a log line bundle on $T$. Then a map $T \to \LogPic_{X/S}[r]$ can be represented by a log line bundle $\ca L$ on $X_T$ such that $\ca L^{\otimes r}$ descends to a line bundle\footnote{It might seem more natural to write `log line bundle on $T$' here. But this is the strict-$S$-etale sheafification of the condition that $\ca L^{\otimes r}$ be trivial. } on $T$. A morphism $(T/S, \ca L) \to (T'/S, \ca L')$ consists of a morphism of log schemes $f\colon T \to T'$ over $X$, such that $\ca L^\vee \otimes f^*\ca L'$ descends to a line bundle on $T$. 
The exact sequence 
\begin{equation}
1 \to \Pic_{X/S}^{\ul 0} \to \LogPic_{X/S} \to \TroPic_{X/S} \to 1
\end{equation}
and the divisibility of $\Pic_{X/S}^{\ul 0}$ in the strict flat topology gives an exact sequence of sheaves in the strict flat topology on $\LogSch_S$:
\begin{equation}
1 \to \Pic_{X/S}^{\ul 0}[r] \to \LogPic_{X/S}[r] \to \TroPic_{X/S}[r] \to 1. 
\end{equation}
This shows that $\LogPic_{X/S}[r]$ is also representable by an algebraic stack with log structure, and that its minimal objects are exactly those that map to minimal objects of $\TroPic[r]$. However, there is an important difference between $\TroPic_{X/S}[r]$ and $\LogPic_{X/S}[r]$: the underlying algebraic stack of the latter is actually an algebraic \emph{space}, i.e. has trivial stabilisers (cf. \ref{eg:P1modGm}), as we now show. 
\begin{proposition}\label{prop:trivial_stabilisers}
The stack $\minCFG{\LogPic_{X/S}[r]}$ is equivalent over $\cat{Sch}_S$ to a category fibred in \emph{setoids} over $\cat{Sch}_S$. 
\end{proposition}
\begin{proof}
Let $(T/S, \ca L)$ be a minimal object of $\LogPic_{X/S}[r]$. The automorphisms of this object over the identity on $\ul T$ in $\cat{Sch}_S$ are morphisms of log schemes $t:T \to T$ satisfying:
\begin{enumerate}
\item $t$ is the identity on the underlying scheme $\ul T$;
\item $t$ commutes with the structure map $\M_S \to \M_T$; 
\item the log line bundle $\ca L^\vee \otimes t^*\ca L$ on $X_T$ descends to a line bundle on $T$. 
\end{enumerate}
We will show that these conditions together force $t$ to be the identity on $T$. 
Now, since $(T, \ca L)$ is minimal over $X$, the ghost sheaf $\ghost_T$ is given by $(t^*\ghost_S)_{\gamma_0/r}$ for some $\gamma_0 \in H_1(\frak X)$. In particular for every $m \in \ghost_T$ we have that $rm$ lies in the image of $\ghost_S \to \ghost_T$. This implies by point (2) that $\overline{t}\colon \ghost_T \to \ghost_T$ is the identity. 
Given $m \in \ghost_T$, we write $\alpha^{-1}m$ for the preimage in $\M_T$, a $\bb G_m$-torsor. Since $t$ acts trivially on $\ghost_T$ it therefore acts on each such torsor; we denote the action by
\begin{equation}
t_m\colon \alpha^{-1}m \to \alpha^{-1}m. 
\end{equation}
We will identify $t_m$ with the unique element of $\bb G_m$ such that the action is multiplication by that element. Condition (2) implies that $t_m = 1$ for every $m \in \ghost_T$ which lies in the image of $\ghost_S \to \ghost_T$. 
Since $t$ acts trivially on $\ghost_S$ we see that the log line bundle $\ca L^\vee \otimes t^*\ca L$ is trivial tropically. This says exactly that it comes from a line bundle of multidegree $\ul 0$ on $X$. To identify this line bundle, let 
\begin{equation}
f: H_1(\frak X) \to {t^*\ghost_S}_{\gamma/r} = \ghost_T; \gamma' \mapsto (0, \gamma')
\end{equation}
be the map of of \ref{eq:f_gamma_over_r_map}, representing the tropicalisation of $\ca L$. We define a map 
\begin{equation}
\phi\colon H_1(\frak X) \to \bb G_m; \gamma \mapsto t_{f(\gamma)}. 
\end{equation}
Then the class of $\ca L^\vee \otimes t^*\ca L$ is given by the image of $\phi$ under the natural injective map 
\begin{equation}
\on{Hom}(H_1(\frak X), \bb G_m) \to \pic^{\ul 0}_{X/S}. 
\end{equation}
In particular, the line bundle $\ca L^\vee \otimes t^*\ca L$ descends to $S$ if and only if $t_{f(\gamma)} = 1$ for every $\gamma \in H_1(\frak X)$. 
Now $\ghost_T = (t^*\ghost_S)_{\gamma_0/r}$ is generated by the image of $\ghost_S$ together with the images of elements of $H_1(\frak X)$ under the map $f$, so we see that $t_m = 1$ for all $m \in \ghost_T$, hence $t$ is the identity. 
\end{proof}
\begin{remark}
In the setup of the proof of the previous proposition, an automorphism $t\colon  T \to T$ satisfying only conditions (1) and (2) of the above proof need \emph{not} be trivial; in this situation the elements $t_m$ need only satisfy $(t_m)^r=1$. This is why $\minCFG{\TroPic_{X/S}[r]}$ is only an algebraic stack with log structure, not an algebraic space; it may admit finite non-trivial stabilizers of order dividing $r$. 
\end{remark}
\begin{example}
Suppose that $S = T = \Spec k$ is a geometric point, with $\ghost_S = \bb N$. Suppose that $X/S$ consists of an irreducible genus 0 curve with a single node of length $1$. The tropical curve $\frak X$ consists of a vertex with a single loop of length $1$; hence $\ca H_1/r\ca H_1$ is the constant group scheme $\bb Z/r\bb Z$ with strict log structure. Suppose for simplicity that $r$ is a prime number; the root stack $\minCFG{\TroPic_{X/S}[r]}\to \ca H_1/r\ca H_1$ looks as follows: it is the identity over the zero section; while over each non-zero $\gamma\in \bb Z/r\bb Z$ it is given by $(B\mu_r,\bb N)\to (\Spec k,\bb N)$, where the map on ghost sheaves is multiplication by $r$. On the other hand, the map $$\minCFG{\LogPic_{X/S}[r]} \to \minCFG{\TroPic_{X/S}[r]}$$ is the identity over the zero section and is the strict map $$(\Spec k, \bb N)\to (B\mu_r, \bb N)$$ on the remaining $r-1$ sections. 
\end{example}
\begin{corollary}\label{cor:LogPicr}
The structure map $\minCFG{\LogPic_{X/S}[r]} \to S$ is representable by proper flat quasi-finite algebraic spaces with log structure. If $\minCFG{\TroPic_{X/S}[r]} \to S$ is strict then the underlying algebraic space of $\minCFG{\LogPic_{X/S}[r]} \to S$ is \'etale and admits a natural group structure. 
\end{corollary}
\begin{proof}
Representability is \ref{prop:trivial_stabilisers}; the other assertions are immediate from \ref{thm:tro_pic_tors_properties} and the short exact sequence \ref{SES:pic_log_tro}. 
\end{proof}
\section{Compactifying spaces of roots of a line bundle}\label{sec:compactifying_roots}
Let $X/S$ be a log curve, $r$ a positive integer, and $\ca L$ on $X$ a line bundle of fibrewise total degree divisible by $r$. In this section we apply results of the previous section to studying the moduli of $r$th roots of $\ca L$. 
\subsection{Compactifying the space of roots of a line bundle}\label{sec:log_descriptions}
\begin{definition}
The \emph{moduli space of logarithmic $r$th roots of $\ca L$}, denoted $\roots[r]{S}{\ca L}$, is the fibre product
\begin{equation}
 \begin{tikzcd}
\roots[r]{S}{\ca L}\arrow[d]  \arrow[r] & \LogPic_{X/S} \arrow[d, "r"] \\
  S \arrow[r, "\ca L"] & \LogPic_{X/S}. 
\end{tikzcd}
\end{equation}
\end{definition}
It carries a universal log curve $X_{\roots{S}{\ca L}}$ by pullback, and the map 
$$\roots[r]{S}{\ca L}\to  \LogPic_{X/S}$$
induces strict \'etale locally on $S$ an equivalence class of log line bundles $\ca L^{\frac{1}{r}}$ such that $(\ca L^{\frac{1}{r}})^{\otimes r} \cong \ca L$. 

\subsubsection{The smooth case}
If $X/S$ is smooth (equivalently, strict) then $\roots{S}{\ca L}$ is exactly the classical space of $r$-th roots of $\ca L$. It is a finite flat cover of $S$ of degree $r^{2g}$, and is \'etale if $r$ is invertible on $S$. The space $\roots{S}{\ca O}$ is a group scheme over $S$, and the spaces $\roots{S}{\ca L}$ are torsors under $\roots{S}{\ca O}$. Locally on $S$ the universal $r$th root is represented by a log line bundle $\ca L^{\frac{1}{r}}$ on $X_{\roots{S}{\ca L}}$ over $\roots{S}{\ca L}$, and this log line bundle is actually a line bundle, and admits an isomorphism
\begin{equation}\label{eq:r-th_power_iso}
(\ca L^{\frac{1}{r}})^{\otimes r} \cong \ca L. 
\end{equation}
\subsubsection{In compact type}
All but the last of the above claims carry over unchanged to the case where $X/S$ is of compact type. The log line bundle $\ca L^{\frac{1}{r}}$ on $X_{\roots{S}{\ca L}}$ can still be represented by a line bundle on $X_{\roots{S}{\ca L}}$, however the equation \ref{eq:r-th_power_iso} holds only on the level of log line bundles. Suppose $S$ is log regular; then this means that (locally on $S$) there exists a vertical Cartier divisor $D$ on $X_{\roots{S}{\ca L}}$ such that 
\begin{equation}
(\ca L^{\frac{1}{r}})^{\otimes r} \cong \ca L(D). 
\end{equation}
This $D$ is chosen so as to make the \emph{multidegree} of $\ca L$ divisible by $r$; this determines $D$ up to addition of a pullback from $S$. 
\subsubsection{Beyond compact type; logarithmic version}\label{sec:log_version}
If $X/S$ is not of compact type then by \ref{cor:LogPicr} the moduli space $\roots[r]{S}{\ca L}$ is a finite flat scheme over $S$ of degree $r^{2g}$ with a log structure, and is log \'etale if $r$ is invertible on $S$. However, it is not in general \'etale even if $r$ is invertible on $S$; in particular this implies that $\roots[r]{S}{\ca O}$ is not in general a group scheme. On the other hand, it is a \emph{logarithmic} group scheme. One way to say this is that the Yoneda functor 
\begin{equation}
\ul{\roots[r]{S}{\ca O}} \colon \cat{Sch}_S^\op \to \cat{Set}
\end{equation}
does not admit a factorisation via the forgetful functor $\cat{Gp} \to \cat{Set}$, but the Yoneda functor
\begin{equation}
\roots[r]{S}{\ca O} \colon \cat{LogSch}_S^\op \to \cat{Set}
\end{equation}
\emph{does} come with a factorisation via the forgetful functor $\cat{Gp} \to \cat{Set}$. 
Alternatively, we might think of our logarithmic group scheme as being given by the data of a multiplication map 
\begin{equation}
m\colon \roots[r]{S}{\ca O} \times_S \roots[r]{S}{\ca O} \to \roots[r]{S}{\ca O}, 
\end{equation}
however taking the underlying scheme does \emph{not} commute with fs logarithmic fibre products; there is a natural map 
\begin{equation}
\ul{\roots[r]{S}{\ca O} \times_S \roots[r]{S}{\ca O}} \to \ul{\roots[r]{S}{\ca O}} \times_\ul{S} \ul{\roots[r]{S}{\ca O}}, 
\end{equation}
but it is not an isomorphism, and there is no way to complete the dashed arrow in the diagram 
\begin{equation}
 \begin{tikzcd}
  \ul{\roots[r]{S}{\ca O} \times_S \roots[r]{S}{\ca O}} \ar[d]  \arrow[r, "\ul{m}"] &  \ul{\roots[r]{S}{\ca O}}\\
  \ul{\roots[r]{S}{\ca O}} \times_\ul{S} \ul{\roots[r]{S}{\ca O}} \arrow[ur, dashed] & 
\end{tikzcd}
\end{equation}
so as to make it commute. If we replace the line bundle $\ca O$ by $\ca L$, and talk about torsors rather than groups, then exactly the same happens; the logarithmic scheme $\roots[r]{S}{\ca L}$ is a (logarithmic) torsor under the logarithmic group scheme $\roots[r]{S}{\ca O}$, but if we take the underlying scheme then it is not a torsor (indeed, there is nothing for it to be torsor under, as $\ul{\roots[r]{S}{\ca O}}$ is not a group). 
\subsection{From logarithmic groups to root stacks}\label{sec:root_stack}
As we have seen, if one is prepared to work in the logarithmic category then $\roots[r]{S}{\ca O}$ has a group structure and $\roots[r]{S}{\ca L}$ a torsor structure (even beyond compact type). However, if one prefers to tarry in the world of schemes then good behaviour of the schemes $\ul{\roots[r]{S}{\ca O}}$ and $\ul{\roots[r]{S}{\ca L}}$ can be restored after passing to a root stack of $S$. 
We define $\ghost_{X/S, r}$ to be the saturation of the image of $\ghost_S$ in 
\begin{equation}
\frac{\ghost_S \oplus (H_1(\frak X) \otimes_{\bb Z} H_1(\frak X))}{(\gamma \cdot \gamma', -r\gamma\otimes\gamma')}, 
\end{equation}
with natural injective map $j\colon \ghost_S \to \ghost_{X/S, r}$. One checks as in \ref{lem:inclusion_integral} that this is an integral system of denominators, hence the root stack $\tilde S \coloneqq \sqrt[j]{S} \to S$ is a proper flat quasi-finite map, and is an isomorphism over the open locus in $S$ over which $X$ is strict. The idea of this root stack is that it `eats up all the ramification' from the spaces $\roots[r]{S}{\ca L}$ (note that $\tilde S$ does not depend on $\ca L$ -- it works for all line bundles of total degree divisible by $r$). The next lemma makes this intuition precise. 
\begin{lemma}\label{lem:strict_after-root}
The natural map from the log fibre product $\roots[r]{S}{\ca L} \times_S \tilde S$ to $\tilde S$ is strict. 
\end{lemma}
\begin{proof}We work strict-\'etale locally, so assume that $X/S$ is nuclear. Since $\LogPic_{X/S}$ is strict over $\TroPic_{X/S}$ it is enough to prove the lemma with $\roots[r]{S}{\ca L}$ replaced by $\TroPic_{X/S}[r]$ (note that every line bundle is tropically trivial). The latter is realised in \ref{lemma_minimal_monoid_integral} as a root stack of the tropical homology $\mathcal H_{1, X/S} = \ca H_1 \to S$. We may work locally on this strict etale cover $\mathcal H_{1}$ of $S$, which comes down to choosing a $1$-cycle $\gamma$ on the tropicalisation $\frak X$ of $X$. Then by \ref{lemma_minimal_monoid_integral} the inclusion $\TroPic_{X/S}[r] \hookrightarrow \mathcal H_1/r\mathcal H_1$ is the root stack $\sqrt[j]{\ca H_1/r\ca H_1}$ with respect to the monoid extension $j\colon \ghost \to \ghost_{\gamma/r}$ of \ref{def:M_gamma_over_r}. 
Note that there is a natural map 
\begin{equation}
\phi\colon \ghost_{\gamma/r} \to \ghost_{X/S, r}; (m, \gamma') \mapsto (m, \gamma \otimes \gamma')
\end{equation}
Unravelling the definitions of root stacks and their minimal objects, we need to show that the induced map
\begin{equation}
\ghost_{\gamma/r} \oplus_{\ghost_S} \ghost_{X/S, r} \to \ghost_{X/S, r}
\end{equation}
(with pushout taken in the category of \emph{sharp} fs monoids) is an isomorphism. This in turn comes down to showing that the kernel of $\phi^\gp$ is torsion. But clearly $\phi$ commutes with the injections from $\ghost_S$ to $\ghost_{\gamma/r} $ and $ \ghost_{X/S, r}$, and the cokernels of these injections are torsion, from which the claim follows. 
\end{proof}
\begin{theorem}
After base-change to $\tilde S$:
\begin{enumerate}
\item
the \emph{underlying scheme} $\ul{\roots[r]{S}{\ca O}}$ is a finite flat group scheme, and is \'etale if $r$ is invertible on $S$;
\item the \emph{underlying scheme} $\ul{\roots[r]{S}{\ca L}}$ is an fppf torsor under $\ul{\roots[r]{S}{\ca O}}$. If $r$ is invertible on $S$ then it is an \'etale torsor (in particular, it is finite \'etale). 
\end{enumerate}
\end{theorem}
\begin{proof}
The corresponding logarithmic statements are in \ref{sec:log_descriptions}. In general these do not imply analogous statements for the underlying schemes, but in the strict case (which applies by \ref{lem:strict_after-root}) they do. 
\end{proof}
\subsection{Examples of the root stack $\tilde S$}\label{sec:examples}
Suppose now that $X/S$ has minimal log structure and that $S$ is log regular; the prototypical example is that $X/S$ is a stable of genus $g$, the induced map $S \to \Mbar_g$ is smooth, and the log structure on $S$ is pulled back from the log structure on $\Mbar_g$ given by the boundary divisor. Suppose that $X/S$ is nuclear (i.e. $S$ is `small enough') with graph $\frak X$, and write $E$ for the set of edges and $E^{ns}$ for the set of non-separating edges. Then $\ghost_{X/S} = \bb N^E$, but $\tilde S$ will only `see' the non-separating edges. 

\subsubsection{Case of a cyclic graph}
First, suppose that $h_1(\frak X) = 1$ (so the non-separating edges all lie on a single cycle), then $\tilde S \to S$ is the root stack associated to the natural map
\begin{equation}\label{eq:small_root_stack}
\ghost_S =  \bb N^E \to \frac{\bb N^E \oplus \bb N}{(\sum_{e \in E^{ns}} \ell(e), 0) \sim (0,r)}
\end{equation}
In terms of classical root stacks of Cartier divisors, this is just root stack $S(\sqrt{\Delta^{ns}}) \to S$, where $\Delta^{ns}$ is the union of the boundary divisors in $S$ corresponding to non-separating edges. 
\subsubsection{Case of a tree with loops attached}
Next suppose that every non-separating edge is a loop (equivalently $h_1(\frak X) = \# E^{ns}$), then $\tilde S \to S$ is the root stack associated to the map 
\begin{equation}
\ghost_S =  \bb N^E \to \bb N^E 
\end{equation}
sending the length $\ell(e)$ of an edge $e$ to $\ell(e)$ if $e$ is separating, and $r\ell(e)$ otherwise. In terms of classical root stacks of Cartier divisors, this is the fibre product 
\begin{equation}\label{eq:big_root_stack}
\bigtimes_{e \in E^{ns}} S(\sqrt{\ell(e)})
\end{equation}
where the fibre product is taken over $S$, and $S(\sqrt{\ell(e)}) \to S$ is the root stack in the irreducible boundary divisor corresponding to $\ell(e)$. 

\subsubsection{Comparison to Chiodo's root stacks}
In the particular case where $S = \Mbar_{g,n}$ we can compare to Chiodo's constructions \cite{Chiodo2008Stable-twisted-}. From these examples we can see that we take a significantly milder root stack of $\Mbar_{g,n}$ than Chiodo, in two ways:
\begin{enumerate}
\item
Chiodo extracts roots at boundary divisors corresponding to both separating and non-separating  nodes, whereas we only extract them for non-separating nodes;
\item Chiodo always extracts a root of each boundary component separately (as in \ref{eq:big_root_stack}), whereas we are often able to extract a milder root stack, as for example in \ref{eq:small_root_stack}. 
\end{enumerate}
To verify the root stacks that Chiodo takes in these examples, the most direct way is to inspect the proof of \cite[Theorem 4.1.4]{Chiodo2008Stable-twisted-}. 

\section{A spin generalisation of the double ramification cycle}

\newcommand{\LogChow}{\on{LogCH}}

If $X/S$ is a smooth proper curve and $\ca L$ on $X$ a line bundle, the \emph{double ramification cycle} $\DR(\ca L)$ measures the locus in $S$ over which $\ca L$ is fibre-wise trivial. Extending this to a log curve $X/S$ has been a somewhat lengthy process (see \cite{Holmes2017Extending-the-d}, \cite{Marcus2017Logarithmic-com}, \cite{Bae2020Pixtons-formula}), and we know now that the DR cycle lives more naturally in the \emph{logarithmic} Chow rings\footnote{Throughout this section our Chow rings are taken in an operational sense, so we do not have to worry about singularities of $S$. } of $S$, defined as the colimit of Chow rings of log blowups of $S$; see \cite{Barrott2019Logarithmic-Cho} for details on log Chow, and \cite{Holmes2017Multiplicativit}, \cite{Molcho2021The-Hodge-bundl}, \cite{Holmes2021Logarithmic-int}, \cite{Molcho2021A-case-study-of} for the resulting logarithmic DR cycle. 

The logarithmic Picard space $\LogPic_{X/S}$ is not algebraic and hence does not have a Chow group; but it does admit a log blowup which is algebraic, and hence has a well-defined log Chow group. The substack of $\LogPic_{X/S}$ consisting of those objects which are isomorphic to pullbacks of log line bundles from $S$ (the `unit section' of $\LogPic_{X/S}$) becomes after suitable log blowup a regularly embedded substack, and hence has a fundamental class $\DR^\LOG \in \LogChow(\LogPic_{X/S})$. 

With the technology of \cite{Molcho2018The-logarithmic} we can give a very concise definition of the double ramification cycle: the line bundle $\ca L$ on $X$ induces a map $\phi_\ca L \colon S \to \LogPic_{X/S}$, and we define $\DR^\LOG(\ca L) \coloneqq \phi_{\ca L}^!\DR^\LOG \in \LogChow(S)$. The intrepid reader may enjoy verifying that this is equivalent to the definition from \cite{Holmes2017Multiplicativit}; the simplest route is via \cite{Holmes2021Logarithmic-int}. The classical DR cycle of \cite{Holmes2017Extending-the-d,Marcus2017Logarithmic-com,Bae2020Pixtons-formula} is then recovered by pushing this logarithmic class down to a classical Chow class on $S$. 

With this approach in hand it is rather easy to define analogues of the DR cycle for $r$-th roots of $\ca L$. The stack $\roots[r]{S}{\ca L}$ comes with a natural map $\phi_{\ca L^{\frac{1}{r}}}\colon \roots[r]{S}{\ca L} \to \LogPic_{X/S}$, and we define 
\begin{equation}
\DR^{\frac{1}{r}, \LOG}(\ca L) = \phi_{\ca L^{\frac{1}{r}}}^!\DR^\LOG \in \LogChow(\roots[r]{S}{\ca L}). 
\end{equation}
This of course can be pushed forward to a classical class 
\begin{equation}\label{eq:spinDR}
\DR^{\frac{1}{r}}(\ca L) \in \on{CH}(\roots[r]{S}{\ca L}). 
\end{equation}

Finally we explain how to specialise the above to moduli of spin curves.  We fix non-negative integers $g$, $n$ and write $\Mbar_{g,n}$ for the stack of stable $n$-marked log curves of genus $g$ --- this is the familiar moduli space of Deligne-Mumford-Knudsen, with log structure coming from the normal-crossings boundary. We denote the markings by $p_1, \dots, p_n$. Choose integers $k$ and $r$ with $r$ positive and dividing $k(2g-2)$, and choose integers $a_1, \dots, a_n$ summing to $k(2g-2)/r$. Then the moduli stack 
\begin{equation}
\roots[r]{\Mbar_{g,n}}{[\omega_{X/S}^{\otimes k}(-\sum_i ra_i p_i)]}
\end{equation}
is a finite flat cover of $\Mbar_{g,n}$ of degree $r^{2g}$, and carries the class of log line bundles $[\ca L] \coloneqq [\omega_{X/S}^{\frac{k}{r}}(-\sum_i a_i p_i)]$. The pullback 
\begin{equation}
\phi_{[\ca L]}^!\DR^\LOG \in \LogChow(\roots[r]{\Mbar_{g,n}}{[\omega_{X/S}^{\otimes k}(-\sum_i ra_i p_i)]})
\end{equation}
is a natural spin-analogue of the log DR cycle, and can as usual be pushed down to $\on{CH}(\roots[r]{\Mbar_{g,n}}{[\omega_{X/S}^{\otimes k}(-\sum_i ra_i p_i)]})$.

%

\section{Formulae for spin DR}\label{sec:DR_appendix}

In this section we write down a formula for the class 
\begin{equation*}
\DR^{\frac{1}{r}}(\ca L) \in \on{CH}(\roots[r]{S}{\ca L}). 
\end{equation*}
of \ref{eq:spinDR}, where $\ca L$ is a line bundle of total degree $0$ (if the total degree is non-zero then the corresponding cycle is simply zero). This will be expressed via the language of piecewise-polynomial functions, as described in \cite{Holmes2021Logarithmic-int,Molcho2021The-Hodge-bundl,Holmes2022Logarithmic-dou}. In brief, (strict) piecewise polynomials on a log stack $T$ are sections of the symmetric algebra $\on{Sym}^\bullet \ghost_T$, and there is a map $\Psi$ taking piecewise polynomials of degree $d$ to elements of $\on{CH}(T)$ of codimension $d$. We will describe $\DR^{\frac{1}{r}}(\ca L)$ in terms of piecewise polynomials on $\roots[r]{S}{\ca L}$. 

While the definition of the class $\DR^{\frac{1}{r}}(\ca L)$ does not require that the universal $r$th root of $\ca L$ be representable globally on $S$ by a log line bundle, the formula we write will require this (just as in \cite{Bae2020Pixtons-formula}). To this end, we will assume that our curve $C/S$ admits a section through the smooth locus (this is essentially always the case when studying DR cycles), so that the universal $r$th root can be represented by a log line bundle trivialised along the chosen section. We will use this log line bundle without further comment in what follows. 

\subsection{Cones of the space of roots}

In order to write down piecewise polynomial functions on $\roots[r]{S}{\ca L}$ we need a fairly explicit description of the cones/charts of $\roots[r]{S}{\ca L}$ in terms of the cones of $S$. To this end, let $s$ be a strict geometric point of $S$ with ghost sheaf $\ghost_s$, and write $\frak X_s$ for the $\ghost_s$-tropical curve (metrised graph) coming from $X_s$. Let $s_r$ be the same geometric point but with an $r$-th root adjoined to each element of $\ghost_s$, and let $\frak X_{s_r}$ be the tropical curve obtained by base-changing to $\ghost_{s_r}$ and then dividing each edge into a chain of exactly $r$ edges of equal length. 

We write $V$ for the set of vertices of $\frak X_s$ and $V_{r}$ for the vertices of $\frak X_{s_r}$, and define $E$ and $E_r$ similarly; in particular
\begin{equation}
\# E_r = r\#E \;\;\; \text{and}\;\;\; \# V_r = \# V + (r-1) \# E. 
\end{equation}
The group of \emph{divisors} $\on{Div}(\frak X_s)$ is the free abelian group on $V$, and we define similarly $\on{Div}(\frak X_{s_r})$. 
We write $\sf{PL} = H^0(\frak X_s, \ghost_s^\gp)$ for the set of piecewise-linear functions, and similarly $\sf{PL}_r = H^0(\frak X_{s_r}, \ghost_{s_r}^\gp)$; equivalently this can be seen as the functions from $V$ (respectively $V_r$) to $\ghost_s$ (respectively $\ghost_{s_r}$) with integer slopes; again, we refer to \cite{Holmes2021Logarithmic-int,Holmes2022Logarithmic-dou,} for more explanation of these notions. If $\alpha$ is a piecewise linear function then we write $\partial \alpha$ for the map taking an oriented edge to the slope of $\alpha$ along that edge. We have `sum of outgoing slopes' maps
\begin{equation}
\nabla\colon \sf{PL} \to \on{Div}(\frak X_s) \;\;\; \text{and} \;\;\; \nabla\colon \sf{PL}_r \to \on{Div}(\frak X_{s_r});
\end{equation}
the images of these maps are called the principal divisors. The multidegree $\ul{\on{deg}}(\ca L)$ can be seen as a divisor on $\frak X_s$, or as a divisor on $\frak X_{s_r}$ supported on the vertices in $V \subseteq V_r$. 


To explicitly describe the the $s_r$-points of $\roots[r]{S}{\ca L}$, we first choose a set of divisors $D$ on the $\frak X_{s_r}$ which are coset representatives for the cosets mapping to $[\ul{\on{deg}}(\ca L)]$ under multiplication by $r$
\begin{equation}
r \colon \frac{\on{Div}(\frak X_{s_r})}{\sf{PL}_r} \to \frac{\on{Div}(\frak X_{s_r})}{\sf{PL}_r}. 
\end{equation}
In other words, each $D$ is a divisor on $\frak X_{s_r}$ such that $rD$ is linearly equivalent to $\ul{\on{deg}}(\ca L)$, and we choose one such $D$ from each linear equivalence class. 


For each $D$ we choose a piecewise linear function $\alpha_D$ such that $rD  = \ul{\on{deg}}(\ca L(\alpha_D))$; this $\alpha_D$ is unique up to addition of a constant. 

\begin{lemma}
Fix $D$ and $\alpha_D$ as above. Then there exists a piecewise linear function $\beta \in \sf{PL}_r$ such that $\alpha + r\beta$ is constant on $V \subseteq V_r$. 
\end{lemma}
\begin{proof}
Fix a vertex $v_0 \in V$; after addition of a constant we may assume that $\alpha(v_0) = 0$. Consider a vertex $v$ connected to $v_0$ by an edge $e \in E$,  which gets divided into $r$ parts in $E_r$; write $v_0 = u_0, u_1, \dots, u_r = v$ for the vertices of $E_r$ along that edge. Then the sum of the outgoing slopes of $\alpha$ vanishes modulo $r$ at each of $u_1, \dots, u_{r-1}$ (since $\ul{\on{deg}}(\ca L)$ is supported on $V$, and $r \mid rD$), so writing $s_i$ for the slope of $\alpha$ from $u_i$ to $u_{i+1}$ we have $r \mid s_i - s_{i+1}$. Then 
\begin{equation}
\begin{split}
\alpha(v) & = \sum_{i=0}^{r-1} s_i \frac{\ell(e)}{r}\\
& = \left(s_0 + \sum_{i=0}^{r-1} \frac{s_i - s_0}{r}\right) r \frac{\ell(e)}{r}\\
\end{split}
\end{equation}
hence $\alpha(v)$ is divisible by $r$ in $\ghost_r$. Applying this to each edge of the graph yields that $r \mid \alpha(v)$ for all $v \in V$, and moreover that for every $e \in E$ with endpoints $v$ and $v'$ that $\ell(e) \mid \alpha(v) - \alpha(v')$. 

Now choose an orientation on each edge in $E$. We define $\beta$ to take the value $-\alpha(v)/r$ for $v \in V$  and for each interior vertex on an edge going \emph{away} from $v$. To see that $\beta$ is piecewise linear, consider an edge $e\in E$ joining vertices $v$ and $v'$; then as noted above we have $\ell(e) \mid \alpha(v) - \alpha(v')$, hence $\ell(e)/r$ (the length of any of the $r$-segments of $e$ in $E_r$) divides the difference in values of $\beta$ between the points of segments. 
\end{proof}
Replacing $\alpha_D$ by $\alpha_D + r\beta$ and $D$ by $D + \nabla\beta$ we may and do assume that $\alpha_D$ takes the value $0$ on each vertex in $V \subseteq V_r$. 

\begin{lemma}\label{lem:compare_reps}
If $(D', \alpha_{D'})$ is another pair with $D$ linearly equivalent to $D'$ and $\alpha_{D'}$ vanishing on $V$, then there exists $\beta\in \sf{PL}_r$ vanishing on $V$ with $D' = D + \on{div}\beta$. 
\end{lemma}
\begin{proof}
We have 
\begin{equation*}
rD = \ul{\on{deg}}\ca L + \nabla \alpha_D \;\;\; \text{and} \;\;\; rD' = \ul{\on{deg}}\ca L + \nabla \alpha_{D'}, 
\end{equation*}
so if $D' = D + \nabla \beta$ then $r \beta - (\alpha_D - \alpha_{D'})$ is constant. Hence $\beta$ is constant on $V$, and so can be chosen to vanish on $V$. 
\end{proof}

Finally, for each $D$ we choose representatives of the isomorphism classes of line bundles $\ca F$ on $X_{s_r}$ with $\ca F^{\otimes r} \isom \ca L(\alpha_D)$. The cones resulting from different choices of $\ca F$ will be canonically isomorphic, so we do not need to distinguish between them when writing a formula below. 

%

\subsection{Piecewise polynomial functions on the cones}
We define our strict piecewise polynomial function locally on $\roots[r]{S}{\ca L}$, so fix a prestable graph $\frak X_s$, a divisor $D$ on $\frak X_{s_r}$, a piecewise linear function $\alpha_D$, and a line bundle $\ca F$ with $\ca F^{\otimes r} \cong \ca L(\alpha_D)$ as above. 

We recall from \cite{Bae2020Pixtons-formula} the class $\eta_\ca L = \pi_*(c_1(\ca L)^2) \in \sf{CH}(S)$, and similarly define
\begin{equation}
\eta_{\ca L(\alpha_D)} \coloneqq \pi_*(c_1(\ca L(\alpha_D))^2). 
\end{equation}
Since we chose $\alpha_D$ to vanish on each vertex $v \in V$ we have $c_1(\ca L) \cdot c_1(\ca O(\alpha_D)) = 0$, so 
\begin{equation}
\eta_{\ca L(\alpha_D)} = \pi_*(c_1(\ca L)^2) + \pi_*(c_1(\ca O(\alpha_D))^2) = \eta_{\ca  L} + \pi_*(c_1(\ca O(\alpha_D))^2). 
\end{equation}
Now $\pi_*(c_1(\ca O(\alpha))^2)$ is the image of the piecewise linear function 
\begin{equation}
\sum_{v \in V_r} \alpha_D(v) \ul{\on{deg}}(\ca O(\alpha_D))(v)
\end{equation}
on $\roots[r]{S}{\ca L}$. Noting that 
\begin{equation}
\ul{\on{deg}}(\ca O(\alpha_D))(v) = \nabla(\alpha_D)(v). 
\end{equation}
we define a piecewise linear function $P_\eta = P_\eta(D)$ on $\roots[r]{S}{\ca L}$ by
\begin{equation}
P_\eta \coloneqq \frac{1}{r^2}\sum_{v \in V_r} \alpha_D(v) \nabla\alpha_D(v). 
\end{equation}

Now for a positive integer $R$ we define a piecewise-polynomial function $f_R(D)$ by 
\begin{equation}
f_R(D) \coloneqq \sum_{w} R^{-h_1(\Gamma)}\prod_{e = (h, h')} \exp(\frac{w(h)w(h')}{2} \ell(e))
\end{equation}
where the sum runs over weightings\footnote{In the sense of \cite{Bae2020Pixtons-formula}: a weighting mod $R$ for $D$ is a function $w$ from half-edges of $\frak X_{s_r}$ to $\{0, \dots, R-1\}$ such that $w(h) + w(h') = 0 \on{mod} R$ if $h$, $h'$ form an edge, and such that the sum of the weights of half-edges at a vertex  $v$ is equal to $-D(v)$ modulo $R$. } $w$ modulo $R$ on $\frak X_{s_r}$ for the divisor $D$, and the product runs over edges $e$ of $\frak X_{s_r}$. 

\begin{lemma}
On quasi-compact opens of $S$, the values of $f_R(D)$ are eventually polynomial in $R$. 
\end{lemma}
\begin{proof}
This can be proven by a slight variant on the arguments in Aaron Pixton's appendix to \cite{Janda2016Double-ramifica}. 
\end{proof}

Hence we can define $P_w = P_w(D)$ to be the value of this polynomial at $R=0$, a piecewise polynomial function on $S$. 
Then for each integer $0 \le d$ we define
\begin{equation}
P_d(D) = [\exp(\frac{-P_\eta(D)}{2})P_w(D)]_d
\end{equation}
where the subscript $d$ means that we take the part in homogeneous degree $d$. 

\begin{lemma}
The piecewise polynomial function $P_d$ is independent of the choice of coset representative $D$. 
\end{lemma}
\begin{proof}
Suppose that $D'$, $\alpha_{D'}$ is another choice of coset representatives. Then by \ref{lem:compare_reps} we know that there exists a piecewise linear function $\beta$ on $\Gamma_r$ vanishing on $V$ and with $D - D' = \on{div} \beta$ and $\alpha' = \alpha + r \beta$. Any such $\beta$ can be formed as an integer linear combination of what we will call \emph{primitive piecewise linear functions} on $\Gamma_r$: given a vertex $u \in V_r \setminus V$ we define a piecewise linear function $\beta_u$ to take value $0$ away from $u$, and to take value $\ell/r$ at $u$, where $\ell$ is the length of the edge of $\Gamma$ on which $u$ lives. Thus it suffices to verify that the piecewise polynomial function $P_d$ does not change when we replace $D$ by $D + \on{div} \beta_u$ and $\alpha $ by $\alpha + r\beta_u$ for some $u \in V_r \setminus V$. In what follows we fix one such $u$, and to simplify notation we write $\beta \coloneqq \beta_u$ and $\alpha \coloneqq \alpha_D$. We write $u_1$ and $u_2$ for the vertices in $V_r$ which lie immediately to either side of $u$. 

We begin by computing the effect of $\beta$ separately on $P_\eta$ and on $P_w$. First, $P_\eta$:
\begin{equation}
\begin{split}
r^2 P_{\eta_{\ca L(\alpha + r \beta)}} & = \sum_{v \in V_r} (\alpha + r \beta)(v) \nabla(\alpha + r \beta)(v)\\
& = \sum_v \alpha(v) \nabla\alpha(v)  + r \sum_v \beta(v) \nabla\alpha(v) + r \sum_v \alpha(v) \nabla\beta (v) + r^2 \sum_v \beta(v) \nabla\beta(v)\\
& = r^2P_{\eta_{\ca L(\alpha)}} + r\frac{\ell}{r} \nabla \alpha(u)  + r(- 2 \alpha(u) + \alpha(u_1) + \alpha(u_2)) - 2 r^2 \frac{l}{r}\\
& = r^2P_{\eta_{\ca L(\alpha)}} - 2\ell(\nabla\alpha(u) - r), \\
& = r^2P_{\eta_{\ca L(\alpha)}} - 2\ell r(D(u) - 1), 
\end{split}
\end{equation}
where for the penultimate equality we use the definition of $\nabla \alpha(u)$: 
\begin{equation}
\nabla\alpha(u) = \frac{\alpha(u_1) - \alpha(u))}{\ell/r} + \frac{\alpha(u_2) - \alpha(u))}{\ell/r}, 
\end{equation}
and for the final equality we use that $rD = \ul{\on{deg}} \ca L + \nabla \alpha$ and that $\ul{\on{deg}}\ca L(u) = 0$ since $u \in V_r \setminus V$. 

Next we compute the effect of addition of $\beta$ on $f_R$, and hence on $P_w$. We have 
\begin{equation}
f_R(D + \beta) \coloneqq \sum_{w} R^{-h_1(\Gamma)}\prod_{e = (h, h')} \exp(\frac{w(h)w(h')}{2} \ell(e))
\end{equation}
where the sum runs over weightings modulo $R$ on $\frak X_{s_r}$ for the divisor $D + \on{div} \beta$. We write $\partial \beta$ for the map assigning to each half-edge the slope of $\beta$ along that half-edge (this has the same data-type as a weighting, but does not in general satisfy its axioms). Now if $w$ is a weighting for $D$ then $w-  \partial\beta$ is a weighting for $D + \on{div} \beta$. Hence
\begin{equation}
f_R(D + \beta) = \sum_{w} R^{-h_1(\Gamma)}\prod_{e = (h, h')} \exp(\frac{(w - \partial \beta)(h)((w - \partial \beta)(h')}{2} \ell(e))
\end{equation}
where $w$ runs over weightings modulo $R$ for $D$. But $\partial \beta$ takes the value 1 on the edge $e_1$ from $u$ to $u_1$ and $e_2$ from $u$ to $u_2$ (both of which have length $\ell/r$), and $0$ on all the other edges. Hence 
\begin{equation}
\begin{split}
f_R(D + \beta) & = \sum_{w} R^{-h_1(\Gamma)}\prod_{e = (h, h')} \exp(\frac{(w - \partial \beta)(h)((w - \partial \beta)(h')}{2} \ell(e))\\
 & = \exp((-w(e_1) - w(e_2) + 1)\ell/r)\sum_{w} R^{-h_1(\Gamma)}\prod_{e = (h, h')} \exp(\frac{w(h)w(h')}{2} \ell(e))\\
 & = \exp((-D(u) +1)\ell/r) f_R(D). 
\end{split}
\end{equation}
Hence we deduce that 
\begin{equation}
P_w(D + \beta) = P_w(D) \exp((-D(u) +1)\ell/r). 
\end{equation}
Putting this together, we find 
\begin{equation}
\begin{split}
\exp(\frac{- P_\eta(D + \beta)}{2}) P_w(D + \beta) & = \exp(\frac{\ell}{r} (D(u) - 1))\exp(- \frac{P_\eta(D)}{2})\exp((-D(u) +1)\ell/r)P_w(D)\\
& = \exp(- \frac{P_\eta(D)}{2})P_w(D)
\end{split}
\end{equation}
as required. 
\end{proof}
This lemma implies that $P_{g-d}$ is a well-defined piecewise polynomial function on $\roots[r]{S}{\ca L}$. Next we define 
\begin{equation}
P = \sum_{d=0}^g \frac{1}{d!}\left(\frac{-\eta_\ca L}{2r^2}\right)^d\Psi(P_{g-d}) \in \on{CH}(\roots[r]{S}{\ca L}), 
\end{equation}
where $\Psi$ is the map from \cite{Holmes2021Logarithmic-int} taking a piecewise polynomial to the corresponding (operational) Chow class. 

\begin{theorem}\label{thm:DR_formula_main}
\begin{equation}
[P]_g = \DR^{\frac{1}{r}}(\ca L) \in \on{CH}(\roots[r]{S}{\ca L}). 
\end{equation}
\end{theorem}
\begin{proof}
If we take $r=1$ then $\roots[r]{S}{\ca L} = S$, and the formula $P$ is the translation of the main formula of \cite{Bae2020Pixtons-formula} into the language of piecewise polynomial functions; see \cite[Theorem 46]{Holmes2022Logarithmic-dou} for details of this translation. The case of any $r \in \bb Z_{>0}$ can be deduced from the case $r=1$ by pullback. Namely, after passing to a proper surjective cover of $\roots[r]{S}{\ca L}$ we may make \emph{canonical} choices of the divisors $D$ (for example, by giving an orientation to each edge of the graphs $\frak X_s$ and requiring the support of $D$ to be moved as far as possible towards the start of the edge). Moreover (perhaps after another such cover) we subdivide each edge of the curve $X$ into a chain of $r-1$ copies of the projective line (so that its graph is $\frak X_{s_r}$). Then $\DR^{\frac{1}{r}}(\ca L)$ is simply the DR cycle associated to the \emph{line bundle} $\ca F$ with $\ca F^{\otimes r} \isom \ca L(\alpha_D)$. To see this, note that DR measures the locus where a log line bundle is trivial, or equivalently the locus where a line bundle is logarithmically trivial, and use that logarithmic triviality is unaffected by passing to subdivisions of the curve by \cite[Theorem 4.4.1]{Molcho2018The-logarithmic}. The formula $P$ is simply the result of applying the $r=1$ case of the formula to the line bundle $\ca F$. 
\end{proof}

\begin{remark}
The extent to which \ref{thm:DR_formula_main} can be considered a `formula' for $\DR^{\frac{1}{r}}(\ca L)$ is of course dependent on how one feels about piecewise polynomial functions. One positive point is that these things are quite easy to compute with; in \cite{Holmes2022Logarithmic-dou} we use this language to give formulae for the log DR cycle, including a SAGE implementation for piecewise polynomials on (blowups of) the moduli space of curves. 
\end{remark}

\bibliographystyle{alpha} 
\bibliography{prebib.bib}
\vspace{+16 pt}
\noindent David~Holmes\\
\textsc{Mathematisch Instituut, Universiteit Leiden, Postbus 9512, 2300 RA Leiden, Netherlands} \\
  \textit{E-mail address}: \texttt{holmesdst@math.leidenuniv.nl}
  
  \noindent Giulio~Orecchia\\
\textsc{White Oak Asset Management, 
Rue du Rhône 13, 1204 Genève,
Switzerland} \\
  \textit{E-mail address}: \texttt{giulioorecchia@hotmail.com}

\end{document}